\documentclass[onefignum,onetabnum]{siamart171218}

\newcommand{\specialcell}[2][c]{\begin{tabular}[#1]{@{}c@{}}#2\end{tabular}}
\usepackage{xmpmulti}

\usepackage{booktabs, arydshln}

\usepackage{lipsum}
\usepackage{amsfonts}
\usepackage{bbm}
\usepackage{graphicx}
\usepackage{epstopdf}
\usepackage{IEEEtrantools}
\usepackage{algorithmic}
\ifpdf
  \DeclareGraphicsExtensions{.eps,.pdf,.png,.jpg}
\else
  \DeclareGraphicsExtensions{.eps}
\fi

\newcommand{\var}{\textrm{Var}}
\newsiamremark{remark}{Remark}
\newsiamremark{hypothesis}{Hypothesis}
\crefname{hypothesis}{Hypothesis}{Hypotheses}
\newsiamthm{claim}{Claim}

\headers{Optimization of quasi-convex function over product measure sets}{J. Stenger, F.Gamboa, M. Keller}

\title{Optimization of quasi-convex function over product measure sets}

\author{J\'er\^ome Stenger\footnotemark[3] \thanks{IMT, Toulouse (\email{fabrice.gamboa@math.univ-toulouse.fr}).}
\and Fabrice Gamboa\footnotemark[1] \thanks{ANITI - Artificial and Natural Intelligence Toulouse Institute.}
\and Merlin Keller\thanks{EDF R$\&$D, Chatou (\email{jerome.stenger@edf.fr}, \email{merlin.keller@edf.fr}).}}

\usepackage{amsopn}

\usepackage{caption}

\ifpdf
\hypersetup{
  pdftitle={Optimization of quasi-convex functions over product measures sets},
  pdfauthor={J. Stenger, F. Gamboa, M. Keller}
}
\fi

\begin{document}

\maketitle

\begin{abstract}
  We consider a generalization of the Bauer maximum principle. We work with tensorial products of convex measures sets, that are non necessarily compact but generated by their extreme points. We show that the maximum of a quasi-convex lower semicontinuous function on this product space is reached on the tensorial product of finite mixtures of extreme points. Our work is an extension of the Bauer maximum principle in three different aspects. First, we only assume that the objective functional is quasi-convex. Secondly, the optimization is performed over a space built as a product of measures sets. Finally, the usual compactness assumption is replaced with the existence of an integral representation on the extreme points. 
  We focus on the product of two different types of measure sets, called the moment class and the unimodal moment classes. The elements of these classes are probability measures (respectively unimodal probability measures) satisfying \textit{generalized} moment constraints. We show that an integral representation on the extreme points is available for such spaces and that it extends to their tensorial product.   
  We give several applications of the Theorem, going from robust Bayesian analysis to the optimization of a quantile of a computer code output. 
 
\end{abstract}



\section{Introduction}
Optimization of convex functions is one of the most studied topic in optimization theory. Indeed, their properties are really interesting especially for minimum search. However, convex functions are also attractive for maximum search. On this matter, the famous Bauer maximum principle \cite{choquet_lectures_1969, kruzik_bauers_2000, alfsen_compact_2012} states that a convex upper semicontinuous function, defined on a compact convex subset of a locally convex topological vector space, reaches its maximum on some extreme points. However, in practice the functions to optimize are rarely convex, and quasi-convex functions are a well-tailored generalization for optimization.
In this paper, we are interested in such quasi-convex functions \cite{penot_quasiconvex_1990}. They are defined on a convex subset $\mathcal{A}$ of a topological vector space, as functions satisfying the inequality $f(\lambda x + (1-\lambda)y)\leq \max\{ f(x), f(y)\}$ for all $x,y\in\mathcal{A}$ and $\lambda\in[0,1]$. Many properties of convex functions have equivalent for quasi-convex functions, we refer to the excellent review on quasi-convex function in \cite{greenberg_review_1971} where a non-exhaustive list is provided. For instance, the Bauer maximum principle remains true for quasi-convex functions. What is more surprising is that the proof of this claim available in \cite[p.102]{choquet_lectures_1969} is similar in all respects to the one for convex functions. 

In this paper, we study a quasi-convex lower semicontinuous function (meaning that $\{ x\in\mathcal{A}: f(x) \leq \alpha\}$ is a closed and convex set for all $\alpha\in\mathbb{R}$) and its optimization on a product space. To our knowledge, this has not been addressed before. Our work is an extension of \cite{owhadi_optimal_2013}, where the authors study the optimization of an affine function on a product of measure spaces with moment constraints. Although our theoretical approach deals with general topological spaces, we focus in our applications on convex sets $\mathcal{A}_1, \dots, \mathcal{A}_d$ of probability measures. We aim to optimize a quasi-convex lower semicontinuous (\textit{lsc}) functional over the product space $\prod_{i=1}^d \mathcal{A}_i$.  
The product of measure sets, also called tensorization, is one of the key elements of this framework. As we will see, it suits numerous industrial optimization problems that will be discussed in this paper.

The strength of our approach is that contrary to the Bauer maximum principle, we do not assume compactness. Instead of the compactness assumption of the optimization sets $\mathcal{A}_i$ for $1\leq i\leq d$, we assume the existence of an integral representation \cite{lukes_integral_2009} by a subset $\Delta_i$ of $\mathcal{A}_i$. This means that every measure $\mu$ in $\mathcal{A}_i$ is the barycenter of a measure $\boldsymbol{\nu}$ supported on $\Delta_{i}$, such that for any linear function $\phi$ of the topological dual of $\mathcal{A}_i$ it holds
\[\langle \phi, \mu\rangle :=\int_X \phi(x)\ \mu(dx) = \int_{\Delta_{i}} \langle \phi , s\rangle \,d\boldsymbol{\nu}(s)\ , \] where $X$ denotes the support of $\mu$. The bold type indicates that $\boldsymbol{\nu}$ is a probability measure supported by a set of probability measures.
The compact case is also included in this framework as the Choquet representation holds. More precisely, every point of a compact convex set is the barycenter of a probability measure carried by every bordering set (see \cite[Theorem 27.6]{choquet_lectures_1969} for details, note that the representation is supported on the set of extreme points if the space is also metrizable \cite[p.140]{choquet_lectures_1969}). The existence of the integral representation is a strong assumption. However, we will provide two different measure spaces for which the integral representation holds; namely the moment class \cite{winkler_extreme_1988, weizsacker_integral_1979} and the unimodal moment class \cite{bertin_unimodality_1997}. 

Our main theorem is therefore an extension of the Bauer maximum principle in three directions: quasi-convexity of the optimization function replaces convexity, tensorization generalizes the structure of the optimization space, and the existence of an integral representation on the marginal sets covers the compact case. By doing so, we build a framework that includes many optimization procedures developed earlier. We refer for example to robust Bayesian analysis \cite{ruggeri_robust_2005, berger_robust_1990, sivaganesan_ranges_1989},  that studies the sensitivity of the Bayesian analysis to the choice of an uncertain prior distribution. Another example is the work in \cite{owhadi_optimal_2013} called Optimal Uncertainty Quantification. Further, we present new applications, all illustrated on a toy case. The theoretical approach is made as general as possible, while the proofs of our claims only rely on simple topological arguments.

The paper is organized as follows. \Cref{sec: Notation and preliminaries} provides the framework basis before introducing our main result in \Cref{sec: Main results}. \Cref{sec: applications} is dedicated to the presentation of some applications that are illustrated in \Cref{sec: case study} on a use case. \Cref{sec: proofs} provides the theoretical formulation and proofs of our results. We discuss in the last section conclusion and perspectives of this work. 

\section{Measure spaces}

\label{sec: Notation and preliminaries}
We will work with a subset of $\mathcal{P}(X)$, the set of all Borel probability measures on a topological space $X$ (specified in the following). Let $C_b(X)$ denote the set of all continuous bounded real valued functions on $X$. We deal with a convex subset of $\mathcal{P}(X)$ satisfying the integral representation property. Note that generally speaking, $\mathcal{P}(X)$ can be considered as a subset of the closed unit ball of the topological dual of $C_b(X)$ and it inherits its topology, which is the topology of weak$^*$ convergence. Moreover, the weak$^*$ topology is always locally convex, since it is induced by the seminorms $\mu\mapsto |\langle\phi, \mu\rangle|$, where $\phi\in C_b(X)$, and $\mu$ belongs to the topological dual of $C_b(X)$.

\subsection{Moment class}
\label{subsec: moment class}
Assume now that $X$ is a Suslin space \cite{Bourles_2018}, i.e. the image of a Polish space under a continuous mapping. We study a convex subspace of $\mathcal{P}(X)$, called the moment class. All measures in the moment class $\mathcal{A}^*$ satisfy \textit{generalized} moment constraints. That is, a measure $\mu\in\mathcal{A}^*$ verifies $\mathbb{E}_{\mu} [\varphi_i] \leq 0$, where $\mathbb{E}_{\mu} [\varphi_i]$ denotes the expectation of $\varphi_i$ with respect to $\mu$, for given measurable functions $\varphi_1,\dots, \varphi_n\in C_b(X)$. Because $X$ is Suslin, all measures $\mu\in\mathcal{P}(X)$ are regular. Hence, the following Theorem \ref{th: extreme points of moments set} due to Winkler \cite[p.586]{winkler_extreme_1988} holds. 
\begin{theorem}[Extreme points of moment class]
  \label{th: extreme points of moments set}
  Consider the space $\mathcal{P}(X)$ of Borel measures on a Suslin space $X$, and measurable functions $\varphi_1, \dots, \varphi_n$ on $X$. Then, for any measure $\mu$ in the moment class $\mathcal{A}^\ast = \{ \mu\in\mathcal{P}(X) \ | \  \mathbb{E}_{\mu} [\varphi_i] \leq 0, 1\leq i\leq n \}$, there exists a probability measure $\boldsymbol{\nu}$ supported on $\Delta^\ast(n)$ such that $\mu$ is the barycenter of $\boldsymbol{\nu}$. Where 
  \[ \Delta^{\ast}({n}) = \left\lbrace \mu\in\mathcal{A}^\ast \, : \ \mu = \sum_{i=1}^{n+1} \omega_i \delta_{x_i},\, \omega_i \geq 0,\, x_i\in X \right\rbrace \ .\]
  is the set of discrete probability measures of $\mathcal{A}^\ast$ supported on at most $n+1$ points.
\end{theorem}
The case of equality in the constraints defining $\mathcal{A}^\ast$ is covered by this result \cite[p.586]{winkler_extreme_1988}.
Theorem \ref{th: extreme points of moments set} states that the extreme points of a class of measures with $n$ \textit{generalized} moment constraints are discrete measures supported on at most $n+1$ points of $X$.

\subsection{Unimodal moment class}
\label{subsec: unimodal class}
In this section, $X$ denotes an interval of the real line $\mathbb{R}$. Let $\mu$ be a probability distribution on $X$, and let $F_\mu$ be its distribution function, that is $F_\mu(h):= \mu(]-\infty,h])$. The measure $\mu$ is said to be unimodal with mode at $a$, if $F_\mu$ is convex on $]-\infty, a[$ and concave on $]a,+\infty[$. We denote $\mathcal{H}_a(X)$ the set of all probability measures on $X$ which are unimodal at $a$. The set $\mathcal{H}_a(X)$ is closed but not necessarily compact ($\mathcal{H}_a(\mathbb{R})$ is not compact, see \cite[p.19]{bertin_unimodality_1997}). Clearly, any uniform probability measure on an interval of the form $ \text{co}({a,z}),\, z\in X$ (co is the convex hull) including the Dirac mass in $a$, is unimodal at $a$. The set $\mathcal{U}_a(X) = \{ \mathfrak{u}\text{ is uniformly distributed on co}(a,z), z\in X \}$ of these uniform probability measures is closed in $\mathcal{P}(X)$ \cite[p.19]{bertin_unimodality_1997}. In this section, we are interested by the convex subset $\mathcal{A}^\dagger$ of unimodal measures satisfying \textit{generalized} moment constraints $\mathbb{E}_\mu[\varphi_i]\leq 0$, for measurable functions $\varphi_1,\dots, \varphi_n$. This subspace is called a unimodal moment class and an equivalent of Theorem \ref{th: extreme points of moments set} holds:
\begin{theorem}[Extreme points of unimodal class]
  \label{th: extreme points of unimodal set}
  Consider the space $\mathcal{H}_a(X)$ of unimodal measures on an interval $X$ with mode $a$, and measurable functions $\varphi_1, \dots, \varphi_n$ on $X$. Then, for any measure $\mu$ in the unimodal moment class $\mathcal{A}^\dagger = \{ \mu\in\mathcal{H}_a(X) \ | \ \allowbreak \mathbb{E}_{\mu} [\varphi_i] \leq 0, 1\leq i\leq n \}$, there exists a probability measure $\boldsymbol{\nu}$ supported on $\Delta^\dagger({n})$ such that $\mu$ is the barycenter of $\boldsymbol{\nu}$. Here
  \[ \Delta^{\dagger}({n}) = \left\lbrace \mu \in\mathcal{A}^\dagger \ | \ \mu = \sum_{i=1}^{n+1} \omega_i \mathfrak{u}_i,\, \omega_i \geq 0,\, \mathfrak{u}_i\in \mathcal{U}_a(X) \right\rbrace \ .\]
  Elements of $\Delta^\dagger(n)$ are mixtures of at most $n+1$ uniform distributions supported on $\textnormal{co}({a,z})$ for some $z\in X$.
\end{theorem}
The proof of this theorem is postponed to the Appendix. The unimodal class was first explored by Khinchin \cite{khintchine1938unimodal} who revealed the fundamental relationship between the set of unimodal probability distributions and uniform probability densities. It was later demonstrated in \cite{bertin_hincspaces_1984} that the Khinchin Theorem may be considered as a non-compact form of the Krein-Milman Theorem \cite[p.105]{choquet_lectures_1969}. In \cite{sivaganesan_ranges_1989}, one can find a first application of this result to the context of robust Bayesian analysis. In the same paper, the class of \textit{symmetric} unimodal distributions with mode $a$ is also considered. They show that the extreme points of this set are mixture of uniform probability measures on an interval of the form $[\theta-z,\theta+z],\, z\geq 0$.

The sets $\mathcal{A}^\dagger$ and $\mathcal{A}^\ast$ are very interesting. Indeed, measure spaces are non-obvious sets and it is generally not straightforward to exhibit their extreme points. 


\section{Main results}
\label{sec: Main results}
\subsection{Construction of the product measure spaces}
\label{subsec: construction of the product measure spaces}
We now give our main Theorem. The measure sets introduced in \Cref{sec: Notation and preliminaries} have very similar properties, so that they are gathered under the same notation. Indeed, we enforce \textit{generalized} moment constraints in both cases. The difference lies in the unimodality of the measures of $\mathcal{A}^\dagger$, while $\mathcal{A}^\ast$ can contain any Radon measure. Consequently, the difference between Theorem \ref{th: extreme points of moments set} and \ref{th: extreme points of unimodal set} lies in the nature of the extreme points. Indeed, the generator of the unimodal moment class $\mathcal{A}^\dagger$ is the set of finite convex combination of uniform distributions. Whereas, the generator of the moment class $\mathcal{A}^\ast$ is the set of finite convex combinations of Dirac masses.

To begin with, we detail the construction of the product space. Let $\mathcal{X}:=X_1\times \dots\times X_d$ be a product of $p$ Suslin spaces $X_1, \dots, X_p$, and $d-p$ real intervals $X_{p+1},\dots, X_d$.
Given real numbers $a_i\in X_i$ for $p< i\leq d$ and measurable functions $\varphi_i^{(j)}: X_i\rightarrow\mathbb{R}$ for $1\leq j \leq N_i$ and $1\leq i\leq d$, we construct $d$ measure spaces with the integral representation property
\begin{IEEEeqnarray*}{lll}
  \mathcal{A}_i = \mathcal{A}_i^\ast= \left\lbrace \mu_i \in\mathcal{P}(X_i)\ | \ \mathbb{E}_{\mu_i}[\varphi_i^{(j)}] \leq 0 \text{ for } j=1,\dots, N_i\right\rbrace & $ for $ 1\leq i\leq p \ ,\\
  \mathcal{A}_i = \mathcal{A}_i^\dagger = \left\lbrace \mu_i \in\mathcal{H}_{a_i}(X_i)\ | \ \mathbb{E}_{\mu_i}[\varphi_i^{(j)}] \leq 0 \text{ for } j=1,\dots, N_i\right\rbrace & $ for $ p< i\leq d \ .
\end{IEEEeqnarray*}
Therefore, the space $\mathcal{A}_i$ is either a moment space on a Suslin space, or a unimodal moment space on an interval as presented in \Cref{sec: Notation and preliminaries}.
We denote by $\Delta_{N_i}\subset \mathcal{A}_i$, the generator of the space $\mathcal{A}_i$, as defined in \Cref{sec: Notation and preliminaries}.  Summarizing, we have
\begin{IEEEeqnarray}{lll}
  \Delta_{i}({N_i}) = \Delta_i^\ast({N_i})=  \left\lbrace \mu_i \in\mathcal{A}_i\ | \ \mu_i = \sum_{k=1}^{N_i+1} \omega_k \delta_{x_k},\, x_k\in X_i \right\rbrace $ for $ 1\leq i\leq p\ ,\\
  \Delta_{i}({N_i}) = \Delta_i^\dagger({N_i})=  \left\lbrace \mu_i \in\mathcal{A}_i\ | \ \mu_i = \sum_{k=1}^{N_i+1} \omega_k \mathfrak{u}_k,\, \mathfrak{u}_k\in\mathcal{U}_{a_i}(X_i),\,\right\rbrace \ $for $ p<i\leq d \ . \nonumber
\end{IEEEeqnarray} 
With these definitions and as discussed in the previous section, any measure $\mu_i \in\mathcal{A}_i$ is the barycenter of a probability measure supported on $\Delta_{i}({N_i})$, the set of convex combination of at most $N_i +1$ Dirac masses or uniform distributions.

In the remaining of the paper, the product spaces $\mathcal{A}=\prod_{i=1}^d \mathcal{A}_i$ and $\Delta = \prod_{i=1}^d \Delta_{i}({N_i})$ are equipped with the product $\sigma$-algebra (not to be confused with the Borel $\sigma$-algebra of the product).

The following definition specifies the meaning of quasi-convexity and lower semicontinuity of a function on a product space.
\begin{definition}
  A function $f: \mathcal{A}\rightarrow \mathbb{R}$ is said to be marginally quasi-convex (marginally \textit{lsc}) if for all $\{\mu_k\in\mathcal{A}_k,\, k\neq i\}$, the function $\mu_i\mapsto f(\mu_1, \dots, \mu_d)$ is quasi convex (respectively \textit{lsc} for the topology of $\mathcal{A}_i$).
\end{definition}
Notice that if $f$ is globally quasi-convex (\textit{lsc} for the product topology) then it is marginally quasi-convex (respectively marginally \textit{lsc}). Indeed, if $f$ is globally \textit{lsc}, then $\{ \mu\in\mathcal{A} \ |\ f(\mu_1, \dots, \mu_d) > \alpha\}$ is open for any $\alpha$ and as the canonical projections are open maps, $\{\mu_i\in\mathcal{A}_i \ | \ f(\mu_1,\dots,\mu_d) > \alpha\}$ is also open. It is clear for quasi-convexity. Having defined properly the product spaces, we give our main result in the next section.

\subsection{Reduction Theorem}
\label{sec: reduction theorem}
We assume that any measure $\mu_i\in \mathcal{A}_i$ belongs to either a moment space or a unimodal moment space. Hence, $\mu_i$ always satisfies $N_i$ moment constraints. In the following theorem (Theorem \ref{th: reduction theorem on product measure space}), we also enforce constraints on the product measure $\mu=\mu_1 \otimes \dots \otimes \mu_d \in\mathcal{A}$, such that, for $N$ measurable functions $\varphi^{(j)}:\mathcal{X}\rightarrow \mathbb{R}$, $1\leq j\leq N$, we have $\mathbb{E}_\mu[\varphi^{(j)}] \leq 0$. In other words, we investigate a subspace of $\mathcal{A}$ defined as $\{\mu\in\mathcal{A}\ | \ \mathbb{E}_\mu[\varphi^{(j)}]\leq 0 , \, 1\leq j\leq N\}$.

\begin{theorem}
  \label{th: reduction theorem on product measure space}
  Suppose that $\mathcal{X}, \mathcal{A}$ and $\Delta$ are defined as in \Cref{subsec: construction of the product measure spaces}. Let $\varphi^{(j)}:\mathcal{X}\rightarrow \mathbb{R},\, 1\leq j\leq N$ be measurable functions. Let $f:\mathcal{A}\rightarrow \mathbb{R}$ be a marginally quasi-convex lower semicontinuous function. Then,
  \[\sup_{\substack{\mu_1,\dots ,\mu_d \in \otimes\mathcal{A}_i \\ \mathbb{E}_{\mu_1,\dots,\mu_d}[\varphi^{(j)}] \leq 0 \\ 1\leq j\leq N}} f(\mu)\quad  = \quad \sup_{\substack{\mu_1,\dots ,\mu_d \in \otimes\Delta_{i}({N_i+N}) \\ \mathbb{E}_{\mu_1,\dots,\mu_d}[\varphi^{(j)}] \leq 0 \\ 1\leq j\leq N}} f(\mu)\]
\end{theorem}

In other words, the supremum of a quasi-convex function on a product space can be computed considering only the $d$-fold product of finite convex combinations of extreme points of the marginal spaces. That is, finite convex combinations of either Dirac masses or uniform distributions.

The underlying assumption is that there exists an integral representation on $\mathcal{A}_i$ for $1\leq i \leq d$. We take two examples of measure spaces whose generators are known, but the proofs in \Cref{sec: proofs} are given in a much more general framework. Notice that this assumption is somehow different from the one used in the Bauer maximum principle, wherein the compactness of the convex space is assumed. However, from the Krein-Millman Theorem \cite[p.105]{choquet_lectures_1969}, the compactness assumption implies the integral representation. Further, since we extend the result to product spaces, our framework is more general. The next proposition highlights that the existence of an integral representation on each marginal space, also implies the existence of an integral representation on the product space.
\begin{proposition}
  Let $\mathcal{X},\mathcal{A}$ and $\Delta$ be defined as in \Cref{subsec: construction of the product measure spaces}, such that the integral representation property holds on every marginal space $\mathcal{A}_i$. Then any measure in $\mathcal{A}$ is also the barycenter of a probability measure supported by $\Delta$.
\end{proposition}
\begin{proof}
  let $\mu$ be in $\mathcal{A}$, so that $\mu = \mu_1\otimes \dots \otimes \mu_d$, with $\mu_i \in\mathcal{A}_i$. Because of the integral representation property of $\mathcal{A}_i$, there exists a probability measure $\boldsymbol{\nu_i}$ supported by $\Delta_i(N_i)$ such that $\mu_i$ is the barycenter of $\boldsymbol{\nu_i}$, i.e. $\mu_i=\int_{\Delta_i(N_i)} s_i\,  d\boldsymbol{\nu}(s_i)$. Therefore, for any function $\phi\in C_b(\mathcal{X})$, using Fubini's Theorem, we have
  \begin{IEEEeqnarray*}{rCl}
    \int_\mathcal{X} \phi(\mathbf{x}) \,\mu(\mathrm{d}\mathbf{x}) & = & \int_\mathcal{X} \phi(x_1,\dots,x_d)\, \mu_1(\mathrm{d}x_1)\dots \mu_d(\mathrm{d}x_d)\ , \\
    & = & \int_\mathcal{X} \phi(x_1,\dots,x_d)\, \int_{\Delta_1(N_1)} s_1(\mathrm{d}x_1)\, \mathrm{d}\boldsymbol{\nu}_1(s_1)\dots \int_{\Delta_d(N_d)} s_d(\mathrm{d}x_d)\, \mathrm{d}\boldsymbol{\nu}_d(s_d)\ , \\
    & = & \int_{\Delta_1(N_1)}\cdots\int_{\Delta_d(N_d)} \int_\mathcal{X} \phi(x_1,\dots,x_d)\,  s_1(\mathrm{d}x_1)\cdots s_d(\mathrm{d}x_d)\, \mathrm{d}\boldsymbol{\nu}_1(s_1)\dots  \mathrm{d}\boldsymbol{\nu}_d(s_d)\ , \\
    & = & \int_{\Delta}\int_\mathcal{X} \phi(\mathbf{x})\,  s(\mathrm{d}\mathbf{x})\, \mathrm{d}\boldsymbol{\nu}(s)\ ,
  \end{IEEEeqnarray*} 
  where $\boldsymbol{\nu} = \boldsymbol{\nu}_1\otimes \dots \otimes \boldsymbol{\nu}_d$ is a probability measure supported on $\Delta$. This means that $\mu$ is the barycenter of $\boldsymbol{\nu}$ in the product space $\mathcal{A}$.
\end{proof}
The proposition above guarantees the existence of the integral representation on the product space $\mathcal{A}$ whenever each marginal space $\mathcal{A}_i$ possesses itself an integral representation property. However, the subspace of $\mathcal{A}$ restricted by moment constraints on the joint distribution $\{\mu\in\mathcal{A}\ | \ \mathbb{E}_\mu[\varphi^{(j)}]\leq 0 , \, 1\leq j\leq N\}$ has a more complex structure than the marginal spaces $\mathcal{A}_i$ from \Cref{sec: Notation and preliminaries}. Indeed, its extreme points are not convex combinations of $N+1$ elements of the generator of $\mathcal{A}$: $\Delta=\bigotimes_{i=1}^d \Delta_{i}(N_i)$. As stated in reduction Theorem \ref{th: reduction theorem on product measure space}, its extreme points are elements of the $d$-fold product of finite convex combinations of extreme points of $\mathcal{A}_i$, that is $\bigotimes_{i=1}^d \Delta_{i}(N_i+N)$.

Notice also that Theorem \ref{th: reduction theorem on product measure space} extends the work of \cite{owhadi_optimal_2013}. Indeed, in this paper the authors were the first to propose the reduction Theorem on a product space. Nevertheless, the optimization considered therein is restricted only to products of moment classes and did not include unimodal moment classes. Moreover, the optimized functional in \cite{owhadi_optimal_2013} is an affine function of the measure. This is a very particular case of our framework. We emphasize that measure affine functions are useful; some of their properties are discussed in the next section.

\subsection{Relaxation of the lower semicontinuity assumption}
\subsubsection{Measure affine functions}
\label{sec: measure affine functions}
The function to be optimized is assumed both lower semicontinuous and quasi-convex. It appears that quasi-convexity covers a large class of functionals fitting most of our application cases. Nevertheless, lower semicontinuity is not always satisfied. So that, it is very interesting to relax this assumption.

In this section, we study specific class of functions that are called measure affine \cite{winkler_extreme_1988}. These functions and their optimization on product measure spaces have been already studied in \cite{owhadi_optimal_2013}. We recall that $ \mathcal{X}, \mathcal{A}$ and $\Delta$ are the product spaces constructed in \Cref{subsec: construction of the product measure spaces}.
\begin{definition}
  A function $F$ is called measure affine whenever $F$ is integrable with respect to any probability measure $\boldsymbol{\nu}$ on $\Delta$ with barycenter $\mu\in\mathcal{A}$ and $F$ fulfills the following barycentrical formula 
  \[ F(\mu) = \int_\Delta F(s)\, d\boldsymbol{\nu}(s)\ .\]
\end{definition}

Notice that any measure affine function satisfies $F(\lambda \mu + (1-\lambda)\pi) = \lambda F(\mu) + (1-\lambda) F(\pi)$, for $\mu,\pi\in\mathcal{A}$ and $\lambda\in[0,1]$. Hence, it is both quasi-convex and quasi-concave. In the following, we show that the optimum of a measure affine function can be computed only on the extreme points of the optimization set, independently of the regularity of $F$. For an extended version enforcing moment constraints on the product measure as in Theorem \ref{th: reduction theorem on product measure space}, we refer to \cite[p.71]{owhadi_optimal_2013}.
\begin{theorem}
  \label{th: reduction theorem for affine measure functional}
  Let $\mathcal{A}$ be a convex subset of a locally convex topological vector space satisfying barycentric property. For any measure affine functional $F$ we have 
  \[\sup_{\mu\in\mathcal{A}} F(\mu) = \sup_{\mu\in\Delta} F(\mu) \ , \]
  and,
  \[ \inf_{\mu\in\mathcal{A}} F(\mu) = \inf_{\mu\in\Delta} F(\mu) \ . \]
\end{theorem}
\begin{proof} The proof is given for the supremum, but it is similar for the infimum. Given $\mu\in\mathcal{A}$, the integral representation property states that there exists a probability measure $\boldsymbol{\nu}$ supported on $\Delta$ such that $\mu$ is the barycenter of $\boldsymbol{\nu}$. Therefore,
  \[ F(\mu) = \int_{\Delta} F(s)d\boldsymbol{\nu}(s) \leq \sup_{s\in\Delta} F(s) \, \]
  for any $\mu\in\mathcal{A}$. Hence, $\sup_{\mu\in\mathcal{A}} F(\mu) \leq \sup_{\mu\in\Delta} F(\mu)$, the converse is obvious as $\Delta \subset \mathcal{A}$.
\end{proof}

\subsubsection{Ratio of measure affine functions}
From the previous theorem, the supremum of a measure affine functional can be searched for only on the generator of the measure space $\mathcal{A}$. We examine some transformation of measure affine functions under which the reduction Theorem \ref{th: reduction theorem on product measure space} still holds and for which the lower semicontinuity assumption remains a non-necessary condition. Ratios of measure affine functions are particularly interesting as it appears in many practical quantities of interest (see for instance \Cref{subsec: sensitivity index} and \Cref{subsec: robust bayesian framework}).

\begin{proposition}
  \label{prop: ratio of measure affine functionals}
  Let $\mathcal{A}$ be a convex set of measures with generator $\Delta$. Let $\phi$ and $\psi$ be two measure affine functionals, $\psi>0$. Then
  \[ \sup_{\mu\in\mathcal{A}} \frac{\phi}{\psi} = \sup_{\mu\in\Delta} \frac{\phi}{\psi} \ , \]
  and,
  \[ \inf_{\mu\in\mathcal{A}} \frac{\phi}{\psi} = \inf_{\mu\in\Delta} \frac{\phi}{\psi} \ .\]
\end{proposition}
\begin{proof} The proof is given for the supremum, but it is similar for the infimum. Given $\mu\in\mathcal{A}$, the integral representation property states that there exists a probability measure $\boldsymbol{\nu}$ supported on $\Delta$ with barycenter $\mu$. Therefore,
  \begin{IEEEeqnarray*}{rCl}
      \phi(\mu) & = & \int_{\Delta} \phi(s) \, d\boldsymbol{\nu}(s)\ , \\
      & = & \int_{\Delta} \frac{\phi(s)}{\psi(s)}\psi(s) \, d\boldsymbol{\nu}(s)\ , \\
      & \leq & \sup_{\Delta}  \frac{\phi}{\psi} \int_{\Delta} \psi(s) \, d\boldsymbol{\nu}(s)\ , \\
      & = &  \sup_{\Delta} \frac{\phi}{\psi} \ \psi(\mu) \ .
  \end{IEEEeqnarray*}
  So that, $\phi(\mu)/ \psi(\mu) \leq \sup_{\Delta} \phi/\psi$ for all $\mu\in\mathcal{A}$, hence $\sup_{\mathcal{A}} \phi/\psi\leq \sup_{\Delta} \phi/\psi$. The other inequality is obvious as $\Delta \subset\mathcal{A}$.
  \end{proof}
Notice that the ratio of a convex function by a positive concave function is quasi-convex \cite[p.51]{Clausing_1983}. Thus, in the previous Proposition the ratio $\phi/\psi$ is quasi-convex.

\section{Applications}
\label{sec: applications}
In this section, we study some practical applications of Theorem \ref{th: reduction theorem on product measure space}, based on real life engineering problems. In the following, we consider a computer code $G$, that can be seen as a black box function. The code $G$ takes $d$ scalar input parameters, that may represent for instance physical quantities. In uncertainty quantification (UQ) methods \cite{ghanem_handbook_2017}, we aim to assess the uncertainty tainting the result of the computer simulation, whose input values are uncertain and modeled as random variables $X_i\sim \mu_i$. They are all assumed to be independent for sake of simplicity. The output of the code $Y=G(X_1,\dots,X_d)$ is therefore also a random variable. Generally, one is interested by the computation of some quantity of interest on the output of the code. However, the choice of the input distributions $(\mu_i)_{\{1\leq i\leq d\}}$ is often itself uncertain. So that, the distributions are often restricted for simplicity to some parametric family, such as Gaussian or uniform. The distribution parameters are then generally estimated with the available information coming from data and/or expert opinion. In practice, this information is often reduced to an input mean value or a variance. We aim to account for the uncertainty on the input distribution choice. So that, we wish to evaluate the maximal quantity of interest over a class of probability distributions. 

In this section, $\mathcal{X}, \mathcal{A}$ and $\Delta$ are constructed as in \Cref{subsec: construction of the product measure spaces}. Therefore, the input distribution $\mu=\mu_1\otimes \dots\otimes \mu_d$ is an element of $\mathcal{A}$, i.e. a product of $d$ independent input measures $\mu_i\in\mathcal{A}_i$. Thus, any input distribution is imprecisely defined and supposed to belong to a moment class or a unimodal moment class.

\subsection{Example of measure affine functions}
\label{subsec: example of measure affine function}
It was shown in \Cref{sec: measure affine functions}, that measure affine functions were particularly interesting because of the relaxation of the lower semicontinuity assumption. Moreover, we have also seen that an affine functional is both quasi-convex and quasi-concave. Hence, it is possible to minimize or maximize the quantity of interest (Theorem \ref{th: reduction theorem for affine measure functional}). We study here specific measure affine functions.
\begin{proposition}
  Let $\mathcal{A}$ be a convex set of probability measure with generator $\Delta$, and let $q_G$ be integrable on $\mathcal{X}$ with respect to any measure in $\Delta$. Then the functional $\mu\mapsto\mathbb{E}_\mu[q_G] = \int_\mathcal{X} q_G\,d\mu$ is measure affine.
\end{proposition}
\begin{proof}
  \begin{IEEEeqnarray*}{rCl}
      \mathbb{E}_\mu[q_G] & = & \int_\mathcal{X} q_G(\mathbf{x}) \,\mu(d\mathbf{x}) \ , \\
      & = & \int_\mathcal{X} q_G(\mathbf{x})\, \int_{\Delta} s(d\mathbf{x})\, d\boldsymbol{\nu}(s) \ ,\\
      & = & \int_{\Delta} \left( \int_{\mathcal{X}}  q_G(\mathbf{x})\, s(d\mathbf{x}) \right) \, d\boldsymbol{\nu}(s) \ , \\
      & = & \int_{\Delta} \mathbb{E}_s[q_G] \, d\boldsymbol{\nu}(s) \ .
  \end{IEEEeqnarray*}  
\end{proof}

The measure affine functional $\mu\mapsto\mathbb{E}_\mu[q_G]$ covers a large range of interest quantities. For instance, the choice $q_G(\mathbf{x})= G(\mathbf{x})$ leads to the expectation of the computer code $G$. Further, any moment can be studied using $q_G(\mathbf{x})=G(\mathbf{x})^n$. The choice $q_G(\mathbf{x})=\mathbbm{1}_{C_G}$, the indicator function on a set $C_G$, yields a probability. An important example would be $q_G(\mathbf{x}) = \mathbbm{1}_{\{G(\mathbf{x})\leq h\}}$, which yields the failure probability at threshold $h\in\mathbb{R}$. The choice of a loss function $q_G(\mathbf{x}) = L(G(\mathbf{x}),a)$ where $a$ is some decision, yields the expected loss associated to the decision $a$.

The interested reader will remark that the lower semicontinuity of the previous affine functional relies on that of $q_G$. More precisely, lower semicontinuity of $q_G$ (respectively upper semicontinuous) implies lower semicontinuity (respectively upper semicontinuity) of the mapping $\mu\mapsto \int q_G\, d\mu$ \cite[Theorem 15.5]{aliprantis_infinite_2006}.

\subsection{Non-Linear Quantities}
\label{subsec: non linear quantities}
We briefly extend the function presented in \Cref{subsec: example of measure affine function} to deal with more general quantities of the form proposed in \cite{berger_robust_1990}. That is
\[\mu\mapsto F(\mu) = \int q_G(\mathbf{x},\varphi(\mu))\, \mu(d\mathbf{x}) \ ,\]
where $\varphi(\mu)$ is measurable. The most common example is:
\[q_G(\mathbf{x},\varphi(\mu)) = (G(\mathbf{x})-\mathbb{E}_\mu[G(\mathbf{x})])^2\ ,  \]
that yields the variance of the distribution $\mu$. In order to compute this quantity, one needs to linearize the problem. The idea is to rewrite the optimization set $\mathcal{A}$ as
\[ \sup_{\mu\in\mathcal{A}} \int q_G(\mathbf{x},\varphi(\mu))\, \mu(d\mathbf{x}) = \sup_{\varphi_0} \sup_{\substack{\mu\in\mathcal{A}\\ \varphi(\mu)={\varphi_0}}} \int q_G(\mathbf{x},\varphi_0)\, \mu(d\mathbf{x})\ .\]
The reduction Theorem \ref{th: reduction theorem on product measure space} applies to the measure affine function $\mu\mapsto  \int q_G(\mathbf{x},\varphi_0)\,\allowbreak \mu(d\mathbf{x})$ on the set $\{ \mu\in\mathcal{A}\ | \ \varphi(\mu)={\varphi_0}\}$, which is the set $\mathcal{A}$ with additional constraints.

\subsection{Quantile Function}
\subsubsection{Lower Quantile Function}
A classical measure of risk widely used in industrial applications \cite{wallis_uncertainty_2004, stenger_optimal_2020}, is the quantile of the output. It is a critical criterion for evaluating safety margins \cite{iooss_advanced_2018}. In the following, $F_\mu$ denotes the cumulative distribution function of the output of the code, i.e. $F_\mu(\alpha)=\mathbb{P}(G(X)\leq \alpha)$.
\begin{theorem}
    We suppose that the code $G$ is continuous. Then the quantile function $\mu\mapsto Q_p^L(\mu)=\inf \{ x : F_\mu(x) \geq p \}$ is quasi-convex and lower semicontinuous on $\mathcal{A}$.
\end{theorem}
\begin{proof}
    A function is quasi-convex if all lower level sets are convex. Further, it is lower semicontinuous if they are closed. Hence, we consider for $\alpha\in\mathbb{R}$ the lower level set for $\alpha\in\mathbb{R}$:
    \begin{IEEEeqnarray*}{rCl}
        L_\alpha & = & \left\lbrace \mu\in\mathcal{A} \ | \ Q_p^L(\mu) \leq \alpha \right\rbrace \ , \\
        & = & \left\lbrace \mu \in\mathcal{A} \ | \ F_\mu(\alpha) \geq p \right\rbrace \ .
    \end{IEEEeqnarray*}
    Indeed, the quantile is the unique function satisfying the Galois inequalities. Therefore,
    \begin{IEEEeqnarray*}{rCl}
        L_\alpha & = & \left\lbrace \mu\in\mathcal{A} \ | \ \mu\left(G^{-1}(]-\infty,\alpha])\right)\geq p \right\rbrace \ .
    \end{IEEEeqnarray*}
    $L_\alpha$ is obviously convex and applying Corollary 15.6 in \cite{aliprantis_infinite_2006}, $L_\alpha$ is also closed (for the weak topology), as $G^{-1}(]-\infty,\alpha])$ is closed .
\end{proof}
Notice that, in this work, the quantile is a function of the measure $\mu$. However, the quantile seen as a function of random variables is not quasi-convex, this subtle point is explained in \cite{drapeau_risk_2012}.

\subsubsection{Upper Quantile Function}
In this paragraph we investigate the minimal value of the quantile of the computer model $G$. In that way, we dispose of bounds around the quantile, which quantifies the range of variation of this QoI over the measure space. However, in order to minimize the quantile, both the upper semicontinuity and quasi-concavity of the optimization function are required. A modified quantity of interest called the upper quantile function is proposed hereunder  \cite{gushchin_integrated_2017}.
\begin{theorem}
    We suppose that the code $G$ is continuous. Then, the upper quantile function $\mu\mapsto Q_p^R(\mu)=\inf \{ x : F_\mu(x) > p \}$ is quasi-concave upper semicontinuous on $\mathcal{A}$.
\end{theorem}
\begin{proof}
    A function is quasi-concave if all upper level sets are convex. It is upper semicontinuous if any upper level set is closed. For $\alpha\in\mathbb{R}$, the upper level set is
    \begin{IEEEeqnarray*}{rCl}
        U_\alpha & = & \left\lbrace \mu\in\mathcal{A} \ | \ Q_p^R(\mu) \geq \alpha \right\rbrace \ , \\[0.3cm]
        & = & \left\lbrace \mu \in\mathcal{A} \ | \ \forall \varepsilon >0 : F_\mu(\alpha-\varepsilon) \leq p \right\rbrace \ , \\[0.3cm]
        & = & \bigcap_{\varepsilon > 0} \left\lbrace \mu \in\mathcal{A} \ | \ F_\mu(\alpha-\varepsilon) \leq p \right\rbrace \ , \\
        & = & \bigcap_{\varepsilon > 0} \left\lbrace \mu \in\mathcal{A} \ | \ \mu(G^{-1}(]-\infty, \alpha-\varepsilon])) \leq p \right\rbrace \ , \\
        & = & \bigcap_{\varepsilon > 0} \left\lbrace \mu \in\mathcal{A} \ | \ \mu(G^{-1}(]-\infty, \alpha-\varepsilon[)) \leq p \right\rbrace \ , 
    \end{IEEEeqnarray*}
    The last two equations differ from the measured interval that is succesively semi-closed then open in the last equation. We prove that this equality holds in two times. For $\varepsilon >0$, we denote 
    \[F_c(\varepsilon)= \left\lbrace \mu \in\mathcal{A} \ | \ \mu(G^{-1}(]-\infty, \alpha-\varepsilon])) \leq p \right\rbrace\ , \]
    which accounts for the semi-closed interval, and
    \[ F_o(\varepsilon)= \left\lbrace \mu \in\mathcal{A} \ | \ \mu(G^{-1}(]-\infty, \alpha-\varepsilon[)) \leq p \right\rbrace \ ,\]
    which is based on the open interval.
    Clearly, $F_c(\varepsilon) \subset F_o(\varepsilon)$ for all $\varepsilon >0$, so that \[\cap_{\varepsilon>0} F_c(\varepsilon) \subset \cap_{\varepsilon>0}  F_o(\varepsilon)\ .\]
    For the reverse inclusion, let $\mu$ be an element of $\cap_{\varepsilon>0}  F_o(\varepsilon)$. Suppose that $\mu$ is not in $\cap_{\varepsilon>0}  F_c(\varepsilon)$. Then, there exists an $\varepsilon_0 >0$ such that $\mu(G^{-1}(]-\infty, \alpha-\varepsilon_0])) > p$. But $\mu(G^{-1}(]-\infty, \alpha-\varepsilon_0])) \leq \mu(G^{-1}(]-\infty, \alpha-\frac{\varepsilon_0}{2}[)) \leq p$, because $\mu$ is in $\cap_{\varepsilon>0}  F_o(\varepsilon)$ by construction, leading to a contradiction. 
    
    To conclude, \cite[Corollary 15.6]{aliprantis_infinite_2006} proves that $F_o(\varepsilon)$ is closed because $G^{-1}(]-\infty, \alpha-\varepsilon[)$ is open as $G$ is continuous. Hence, $U_\alpha$ is closed as an intersection of closed sets. $U_\alpha$ is also obviously convex.
\end{proof}

\subsection{Sensitivity index} 
\label{subsec: sensitivity index}
Global sensitivity analysis aims at determining which uncertain parameters of a computer code mainly drive the output. In that matter, Sobol' indices are widely used as they quantify the contribution of each input to the variance of the output of the model \cite{iooss_review_2015}. However, because the probability distributions modeling the uncertain parameters are themselves uncertain, we propose to evaluate bounds on the Sobol' indices over a class of probability measures. We will focus for simplicity on the well-known first order sensitivity index:
\[S_i(\mu_i) = \frac{\var_{\mu_i}(\mathbb{E}_{\sim i}[Y | X_i])}{\var(Y)}\ , \]
where $\mathbb{E}_{\sim i}$ denotes the expectation over all but the $i$th input variable. The total-effect index $S_{Ti}$ \cite{iooss_review_2015} could be processed in the same way (see Equation \eqref{eq: total order}). 
\begin{theorem}
  Let $\mathcal{X}, \mathcal{A}$ and $\Delta$ be defined as in \Cref{subsec: construction of the product measure spaces}. Then
  \[ \sup_{\mu_i\in\mathcal{A}_i} S_i({\mu}_i) = \sup_{\mu_i\in\Delta_{i}(N_i+1)} S_i(\mu_i) \ , \] 
  \[ \inf_{\mu_i\in\mathcal{A}_i} S_i({\mu}_i) = \inf_{\mu_i\in\Delta_{i}(N_i+1)} S_i(\mu_i) \ . \]
\end{theorem}
\begin{proof}
  The proof is made for the supremum but is similar for the infimum
  \begin{IEEEeqnarray*}{rCl}
    \sup_{\mu_i\in\mathcal{A}_i} S_i(\mu_i) & = & \sup_{\mu_i\in\mathcal{A}_i} \frac{\mathbb{E}_{\mu_i}\left[ \left(\mathbb{E}_{\sim i}[Y|X_i]\right)^2 \right] - \left(\mathbb{E}_{\mu_i}\left[ \mathbb{E}_{\sim i}[Y|X_i] \right] \right)^2 }{\mathbb{E}_{\mu_i}\left[ \mathbb{E}_{\sim i}[Y^2]\right] - \left(\mathbb{E}[Y] \right)^2 } \ , \\[0.4cm]
    & = & \sup_{\mu_i\in\mathcal{A}_i} \frac{\mathbb{E}_{\mu_i}\left[ \left(\mathbb{E}_{\sim i}[Y|X_i]\right)^2 \right] - \left(\mathbb{E}[Y] \right)^2 }{\mathbb{E}_{\mu_i}\left[ \mathbb{E}_{\sim i}[Y^2]\right] - \left(\mathbb{E}[Y] \right)^2 } \ , \\[0.4cm]
    & = & \sup_{\overline{y}_0} \sup_{\substack{\mu_i\in\mathcal{A}_i \\ \mathbb{E}[Y]=\overline{y}_0 }} \frac{\mathbb{E}_{\mu_i}\left[ \left(\mathbb{E}_{\sim i}[Y|X_i]\right)^2 \right] - \overline{y}_0^2 }{\mathbb{E}_{\mu_i}\left[ \mathbb{E}_{\sim i}[Y^2]\right] - \overline{y}_0^2 } \ ,
\end{IEEEeqnarray*}
where $\overline{y}_0$ is a real. Now, the function 
\[\mu_i\longmapsto \frac{\mathbb{E}_{\mu_i}\left[ \left(\mathbb{E}_{\sim i}[Y|X_i]\right)^2 \right] - \overline{y}_0^2 }{\mathbb{E}_{\mu_i}\left[ \mathbb{E}_{\sim i}[Y^2]\right] - \overline{y}_0^2 } \ , \]
is a ratio of two measure affine functionals. Proposition \ref{prop: ratio of measure affine functionals} states that the reduction Theorem \ref{th: reduction theorem on product measure space} applies so that
\[ \sup_{\substack{\mu_i\in\mathcal{A}_i \\ \mathbb{E}[Y]=\overline{y}_0 }} \frac{\mathbb{E}_{\mu_i}\left[ \left(\mathbb{E}_{\sim i}[Y|X_i]\right)^2 \right] - \overline{y}_0^2 }{\mathbb{E}_{\mu_i}\left[ \mathbb{E}_{\sim i}[Y^2]\right] - \overline{y}_0^2 }  =  \sup_{\substack{\mu_i\in\Delta_{i}(N_i+1) \\ \mathbb{E}[Y]=\overline{y}_0 }} \frac{\mathbb{E}_{\mu_i}\left[ \left(\mathbb{E}_{\sim i}[Y|X_i]\right)^2 \right] - \overline{y}_0^2 }{\mathbb{E}_{\mu_i}\left[ \mathbb{E}_{\sim i}[Y^2]\right] - \overline{y}_0^2 } \ ,  \]
and the result follows.
\end{proof}

\subsection{Robust Bayesian framework}
\label{subsec: robust bayesian framework}
Robust Bayesian analysis \cite{ruggeri_robust_2005} studies the influence of the choice of an uncertain prior distribution. The answer is robust if it does not depend significantly on the choice of the inputs prior distributions. Therefore, a Bayesian analysis is applied to all possible prior distributions in a given class of measures.

The posterior probability distribution can be calculated with Bayes' Theorem by multiplying the prior probability distribution $\pi$ by the likelihood function $\theta\mapsto l(x\,|\,\theta)$, and then dividing by the normalizing constant, as follows:
\begin{IEEEeqnarray*}{rCl}
    l(\theta\,|\, x) & = & \frac{l(x\,|\,\theta)\pi(\theta)}{\int l(x\,|\, \theta)\, \pi(d\theta)}    
\end{IEEEeqnarray*}    
Thus, it is natural to define $\Psi$ that maps the prior probability measure to the posterior probability measure. In what follows, $X$ denotes a Polish space
  \begin{IEEEeqnarray*}{rrClrCl}
      \Psi :\quad & \mathcal{P}({X}) & \longrightarrow & \mathcal{P}({X})  \\
      & \pi & \longmapsto & \Psi(\pi) :\quad & C_b({X}) & \longrightarrow & \mathbb{R} \\
      & & & & q & \longmapsto & \Psi(\pi)(q) = \frac{\int_{X} q(\theta) l(x|\theta)\pi(d\theta)}{\int_{X} l(x|\theta) \pi(d\theta)}
  \end{IEEEeqnarray*}
The functional $\Psi$ has very useful properties:
\begin{lemma}
    \label{lem: continuity of the Bayesian transportation map}
    If the likelihood function $l(x|\cdot):\theta \mapsto l(x|\theta)$ is continuous, then $\Psi$ is continuous for the weak$^*$ topology in $\mathcal{P}(\mathcal{X})$.
\end{lemma}
\begin{proof}
    let ($\pi_n$) be a sequence of probability measure in $\mathcal{P}(X)$ converging in weak$^\star$ topology towards some probability measure $\pi$. The convergence in weak$^\star$ topology means that for every $q\in C_b(X)$, $\langle \pi_n | q\rangle \rightarrow \langle \pi |q\rangle$.
    But because $l(x|\cdot)$ is continuous the function $q\times l(x|\cdot)$ is also an element of $C_b(X)$, therefore
    \[\int_X q(\theta)l(x|\theta)\, \pi_n(d\theta) = \langle q\times l(x|\cdot)|\pi_n\rangle \longrightarrow \langle q\times l(x|\cdot)|\pi\rangle = \int_X q(\theta)l(x|\theta)\, \pi(d\theta)\ , \]
    This exactly means that $\Psi(\pi_n)$ converges to $\Psi(\pi)$ in the weak$^*$ topology. This implies the sequential continuity of $\Psi$, thus its continuity. Indeed, as $X$ is polish it is separable and metrizable. So that, $\mathcal{P}(X)$ is also metrizable \cite[Theorem 15.12]{aliprantis_infinite_2006}. Hence, it is first-countable \cite[Theorem 4.7]{croom_principles_2016} which implies it is also sequential. This means that sequential continuity is equivalent to continuity. 
\end{proof}
The function $\Psi$ can be decomposed as a ratio $\Psi=\Psi_1 / \Psi_2$, with
\begin{IEEEeqnarray*}{rrClrCl}
  \Psi_1 :\quad & \mathcal{P}({X}) & \longrightarrow & \mathcal{M}({X})  \\
  & \pi & \longmapsto & \Psi_1(\pi) :\quad & C_b({X}) & \longrightarrow & \mathbb{R} \\
  & & & & q & \longmapsto & \Psi_1(\pi)(q) = \int_{X} q(\theta) l(x|\theta)\pi(d\theta) \ .
\end{IEEEeqnarray*}
where $\mathcal{M}(X)$ denotes the set of all Borel measures on $X$, and
\begin{IEEEeqnarray*}{rrClrCl}
  \Psi_2 :\quad & \mathcal{P}({X}) & \longrightarrow & \mathbb{R}_{^+}^*  \\
  & \pi & \longmapsto & \int_{X} l(x|\theta) \pi(d\theta) \\[-0.9cm] 
  & & & & \color{white} q & \color{white} \longmapsto & \color{white}\Psi_1(\pi)(q) =\color{white} \int_{X} q(\theta) l(x|\theta)\pi(d\theta) \ .
\end{IEEEeqnarray*}
The main property of $\Psi_1$ and $\Psi_2$ is that they are linear maps. The posterior distribution is therefore the ratio of two linear functions of the prior density. This is particularly interesting due to the following Proposition, which states that the composition of a quasi-convex function with the ratio of two linear mapping is also quasi-convex.
\begin{proposition}
    \label{prop: Bayesian quasi-convex composition}
    Let $\mathcal{A}$ be a convex subsets of a topological vector space, and $f$ be a quasi-convex lower semicontinuous functional on $\mathcal{A}$. If $\Psi_1:\mathcal{A} \mapsto \mathcal{A} $ is a linear mapping and $\Psi_2: \mathcal{A} \mapsto \mathbb{R}_{^+}^*$ is a linear functional. Then, $f\circ(\Psi_1/\Psi_2) : \mathcal{A}\mapsto\mathbb{R}$ is also a quasi-convex lower semicontinuous functional.
    \label{lem: composition of ratio of linear functional and quasiconvexity}
\end{proposition}
\begin{proof}
    Let $\pi, \mu$ be in $\mathcal{A}$. Given $\lambda \in[0,1]$, notice that
    \begin{IEEEeqnarray*}{rCl}
        f\left( \frac{\Psi_1(\lambda \pi + (1- \lambda) \mu ) }{\Psi_2(\lambda \pi + (1- \lambda) \mu)}\right) & = & f\left( \frac{\lambda \Psi_1(\pi) + (1- \lambda) \Psi_1(\mu) ) }{\lambda \Psi_2(\pi) + (1- \lambda) \Psi_2(\mu) }\right)\ , \\ 
        \\
        & = & f\left(\beta  \frac{ \Psi_1(\pi) }{\Psi_2(\pi)} + (1-\beta)  \frac{ \Psi_1(\mu) }{\Psi_2(\mu)}\right) \ ,
    \end{IEEEeqnarray*}
    with $\displaystyle \beta = \frac{\lambda \Psi_2(\pi)}{\lambda \Psi_2(\pi) + (1-\lambda) \Psi_2(\mu)}$ in $[0,1]$. Hence,
    \begin{IEEEeqnarray*}{rCl}
        f\left( \frac{\Psi_1(\lambda \pi + (1- \lambda) \mu ) }{\Psi_2(\lambda \pi + (1- \lambda) \mu)}\right)  & \leq & \max\left\lbrace f\left(\frac{\Psi_1(\pi) }{\Psi_2(\pi)}\right)\ ;  \ f\left(\frac{\Psi_1(\mu) }{\Psi_2(\mu)}\right)\right\rbrace \ . \\    
    \end{IEEEeqnarray*}
    This proves the quasi-convexity of $f\circ(\Psi_1/\Psi_2)$. The lower semicontinuity stands because for $\alpha\in\mathbb{R}$, the lower level set 
    \[ \Gamma_\alpha = \left\lbrace \mu\in\mathcal{A} \ | \ f\left(\frac{\Psi_1(\mu)}{\Psi_2(\mu)}\right) \leq \alpha \right\rbrace = \left\lbrace \mu \ | \ \frac{\Psi_1(\mu)}{\Psi_2(\mu)} \in f^{-1}\left(]-\infty,\alpha]\right) \right\rbrace \ ,\] is the inverse image of the lower level set $\alpha$ under the continuous map $\mu\mapsto \Psi_1(\mu)/\Psi_2(\mu)$ according to Lemma \ref{lem: continuity of the Bayesian transportation map}. Therefore, $\Gamma_\alpha$ is closed.
\end{proof}
Proposition \ref{lem: composition of ratio of linear functional and quasiconvexity} proves that any lower semicontinuous quasi-convex function presented above is well suited for robust Bayesian analysis. Therefore, the input distribution of a computer model $\mu$ can derive from a Bayesian inference such that $\mu = \Psi_1(\pi)/ \Psi_2(\pi)$, where $\pi$ is an imprecise distribution modeled in a reasonable class of prior. 
For instance, the optimization of the quantile of posterior distributions inferred from priors in a moment class can be reduced to the extreme points of this class.

Moreover, one can easily see that if the functional $f$ is measure affine then $f\circ(\Psi_1/\Psi_2)$ is also the ratio of two measure affine functionals. From Proposition \ref{prop: ratio of measure affine functionals}, it then holds that lower semicontinuity is not necessary to apply the reduction theorem. This means that we can optimize moments or probabilities of the posterior distribution over a class of prior distributions \cite{sivaganesan_ranges_1989}.

\section{Application to a use case}
\label{sec: case study}

To illustrate our theoretical optimization results, we address a simplified hydraulic model \cite{pasanisi_estimation_2012}. This code calculates the water height $H$ of a river subject to a flood event. The height of the river $H$ is calculated through the analytical model given in Equation \eqref{eq:Hydraulic Model}.

\begin{minipage}[c]{0.45\linewidth}
    \includegraphics[scale=0.6]{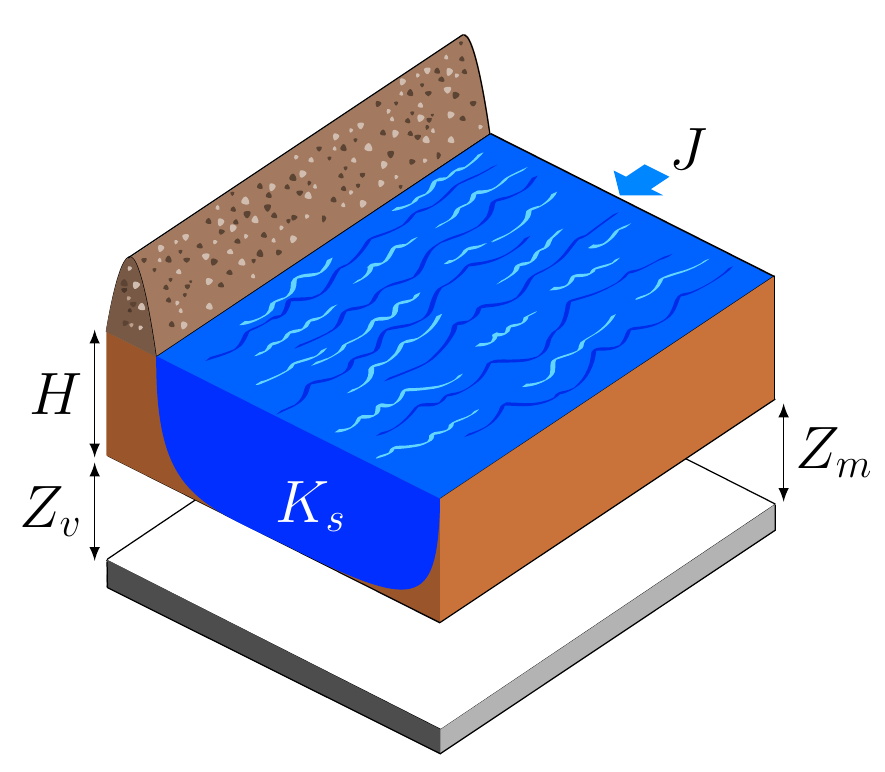}
    \captionof{figure}{Simplified river cross section.}
    \label{fig: River cross section}
\end{minipage}\hfill
\begin{minipage}[c]{0.48\linewidth}
  \begin{equation}
	H = \left(\frac{J}{300K_s \sqrt{\frac{Z_m - Z_v}{5000}}} \right)^{3/5} \ .
	\label{eq:Hydraulic Model}
\end{equation}
\vspace{0.5cm}
\end{minipage}
The code takes four inputs whose initial joint distributions are detailed in Table \ref{tab: Initial distribution hydraulic model}. The choice of uniform distributions for $Z_v$ and $Z_m$ comes from expert opinions. The normal distribution for $K_s$ models the uncertainty of the value of the empirical Manning-Strickler coefficient. At last, the choice of a Gumbel distribution is due to the extreme nature of the flooding event. We compute the maximum likelihood parameters of the Gumbel distribution based on a sample of 47 annual maximal flow rates.
\begin{table}[htb]
	\centering	\caption{Initial distribution of the 4 inputs of the hydraulic model.}
	\label{tab: Initial distribution hydraulic model}
	\begin{tabular}{lrr}
		\toprule
		Variable & Description & Distribution \\ 
		\cmidrule(lr){3-3}\cmidrule(lr){2-2}
		$J$ & annual maximum flow rate & $Gumbel(\rho = 626, \beta =190)$ \\
		$K_s$ & Manning-Strickler coefficient & $\mathcal{N}(\overline{x}=30,\sigma=7.5)$ \\
		$Z_v$ & Depth measure of the river downstream & $\mathcal{U}(49,51)$ \\
		$Z_m$ & Depth measure of the river upstream & $\mathcal{U}(54,55)$  \\
		\bottomrule
	\end{tabular}
\end{table}

Notice that the modelization of the parameters through the distribution given in Table \ref{tab: Initial distribution hydraulic model} is questionable. Therefore, as we desire to relax the choice of a specific distribution, we evaluate robust bounds over a measure space taken as a moment class. We display in Table \ref{tab : Constraints for hydraulic model} the corresponding moment constraints that the variables must satisfy. These constraints are calculated based on a sample of 47 annual flow rates and expert opinions. The bounds are taken in order to match the most acceptable values of the parameters. Notice that the distribution of $K_s$ belongs to a unimodal moment class because we consider that there is a most significant value for this empirical constant. Using the previous notations, the input distribution $\mu \sim (J,K_s,Z_v,Z_m)$ belongs to $\mathcal{A} = \mathcal{A}_1^\ast \otimes \mathcal{A}_2^\dagger \otimes \mathcal{A}_3^\ast \otimes \mathcal{A}_3^\ast $ with 
\begin{equation}
  \label{eq: use case measure space definition}
  \begin{IEEEeqnarraybox}[][c]{rCl}
    \mathcal{A}_1^\ast & = & \left\lbrace \mu_1\in\mathcal{P}([160, 3580]) \ | \ \mathbb{E}_{\mu_1}[X] = 736, \ \mathbb{E}_{\mu_1}[X^2]= 602043  \right\rbrace \ , \\
    \mathcal{A}_2^\dagger & = & \left\lbrace \mu_2\in\mathcal{H}_{30}([12.55, 47.45]) \ | \ \mathbb{E}_{\mu_2}[X] = 30,\ \mathbb{E}_{\mu_2}[X^2] = 949 \right\rbrace \ , \\
    \mathcal{A}_3^\ast & = & \left\lbrace \mu_3\in\mathcal{P}([49, 51]) \ | \ \mathbb{E}_{\mu_3}[X] = 50 \right\rbrace \ , \\
    \mathcal{A}_4^\ast & = & \left\lbrace \mu_4\in\mathcal{P}([54, 55]) \ | \ \mathbb{E}_{\mu_4}[X] = 54.5 \right\rbrace \ .
  \end{IEEEeqnarraybox}
\end{equation}
and respective extreme point sets $\Delta_1^\ast(2), \Delta_2^\dagger(2), \Delta_3^\ast(1),$ and $\Delta_4^\ast(1)$.
\begin{table}[ht]
	\centering\caption{Corresponding moment constraints of the 4 inputs of the hydraulic model.}
	\label{tab : Constraints for hydraulic model}
	\begin{tabular}{lrrrr}
		\toprule
		Variable & Bounds & Mean & \specialcell{Second order \\ moment} & Mode \\
		\cmidrule(lr){2-2}\cmidrule(lr){3-3}\cmidrule(lr){4-4}\cmidrule(lr){5-5}
    n$^\circ 1: J$ & $[160, 3580]$ & $736$ & $602043$ & $-$ \\
		n$^\circ 2: K_s$ & $[12.55, 47.45]$ & $30$ & $949$ & $a=30$\\
		n$^\circ 3: Z_v$ & $[49,51]$ & $50$ & $-$ & $-$\\
		n$^\circ 4: Z_m$ & $[54,55]$ & $54.5$ & $-$ & $-$\\
		\bottomrule
	\end{tabular}
\end{table}

The following results have been computed in Python. The quantity of interest is optimized globally using a differential evolutionary algorithm \cite{price_differential_2005} over the set of extreme points $\Delta=\Delta_1^\ast(2)\otimes\Delta_2^\dagger(2)\otimes\Delta_3^\ast(1)\otimes \Delta_4^\ast(1)$. This set is parametric as each element is a probability measure written as a product of mixture of Dirac masses and mixture of uniform distributions. In order to explore this space, we rely on a well suited parameterization based on canonical moments \cite{dette_theory_1997}. Indeed, it has been proven in $\cite{stenger_optimal_2020}$ that there exists a bijection between the set of discrete probability measures supported on at most $n+1$ points satisfying $n$ moment constraints and a pavement of the form $[0,1]^{n+1}$. This improves the exploration efficiency of the global optimizer and the overall performance of the optimization.

\subsection{Computation of failure probabilities and quantiles}

The study addresses the height choice of the protection dike in terms of cost and security.
In order to provide safety margins that are optimal with respect to the uncertainty tainting the inputs distributions, we compute the maximal $p$-quantile over the moment class defined through the constraints in Table \ref{tab : Constraints for hydraulic model}.
The computation of the maximum quantile is equivalent to the computation of the lowest failure probability $\inf_{\mu\in\mathcal{A}} F_\mu$ over the same measure space. Indeed, we have the following duality transformation \cite{stenger_optimal_2020}:
\[ \sup_{\mu \in \mathcal{A}} Q_p^L(\mu) = \inf \left\lbrace h  \in \mathbb{R}\; | \; \inf_{\mu\in\mathcal{A}} F_{\mu}(h) \geq p \right\rbrace\ . \]

The results are depicted in Figure \ref{fig: robust optimization over measure space}. The quantile of order $0.95$ is equal to $2.75m$ for the initial distribution, which gives the appropriate safety margins needed to build a protection dike. However, by considering the uncertainty tainting the input distribution contained in the class $\mathcal{A}$ (the dashed line), the maximum $0.95$-quantile over this class is equal to $3.05m$.

\subsection{Computation of a Bayesian quantity of interest}
We now consider that $J$ is modeled as initially with a Gumbel distribution (see Table \ref{tab: Initial distribution hydraulic model}). Indeed, extreme value theory \cite{coles_introduction_2001} justifies the choice of a Gumbel distribution for the maximal annual flow rate. However, in a Bayesian setting, the location parameter $\rho$, and the scale parameter $\beta$ of the Gumbel distribution are associated with a prior distribution $\pi(\rho, \beta)$. In \cite{pasanisi_estimation_2012}, the prior distribution was taken to be low informative using $\rho\sim \mathcal{G}(1, 500)$ and $1/\beta \sim \mathcal{G}(1,200)$, where $\mathcal{G}(\alpha,\tau)$ is the Gamma distribution with shape parameter $\alpha$ and scale parameter $\tau$.

This choice of prior is questionable. Instead, we used the previously computed maximum likelihood estimation as a mean constraint. The bounds are chosen as reasonable values.
\begin{table}[ht]
	\centering\caption{Corresponding moment constraints of the parameters $\rho, \beta$ of the Gumbel distribution of $J$.}
	\label{tab : Constraints for Bayesian parameters}
	\begin{tabular}{lrr}
		\toprule
		Variable & Bounds & Mean \\
		\cmidrule(lr){2-2}\cmidrule(lr){3-3}
		$\rho$ & $[550, 700]$ & $626.14$  \\
		$\beta$ & $[150, 250]$ & $190$  \\
		\bottomrule
	\end{tabular}
\end{table}

\noindent
This corresponds to two moment classes, $\rho$ belongs to $\widetilde{\mathcal{A}}_1^{\ast} = \{ \mu\in\mathcal{P}([550,700]) \ | \ \mathbb{E}_{\mu}[X]\allowbreak = 626.14 \}$ and $\beta$ to $\widehat{\mathcal{A}}_1^{\ast} = \{ \mu\in\mathcal{P}([150,250]) \ | \ \mathbb{E}_{\mu}[X] = 190 \}$. Distributions of the other parameters $K_s,  Z_v, Z_m$ are constrained to their previous classes in Equation \eqref{eq: use case measure space definition}, that is respectively $\mathcal{A}_2^\dagger, \mathcal{A}_3^\ast$ and $\mathcal{A}_4^\ast$.
Finally, the distribution  $\Theta\sim (\rho, \beta, K_s, Z_v, Z_m)$ belongs to the product space $\mathcal{A}' = \widetilde{\mathcal{A}}_1^{\ast}\otimes \widehat{\mathcal{A}}_1^{\ast} \otimes \mathcal{A}_2^\dagger \otimes \mathcal{A}_3^\ast \otimes \mathcal{A}_4^\ast $.

The Gumbel model and the analytic formulation of the code in Equation \ref{eq:Hydraulic Model} yields the exact calculation of the probability of failure conditional on $ (\rho, \beta, K_s, Z_v, Z_m)$:
\[\mathbb{P}(H \leq h \ |\ \Theta) = \exp \left(- \exp \left\lbrace \beta \left( \rho - 300K_s \sqrt{\frac{Z_m-Z_v}{5000}} (h - Z_v)^{5/3}\right) \right\rbrace \right) \ .  \]
Therefore, the Bayesian probability of failure corresponds to the integrated cost
\begin{equation}
  F_\Theta(h) = \mathbb{P}(H\leq h) = \int \mathbb{P}(H \leq h \ |\ \Theta)\, \pi(\Theta | D)\, d\Theta\ ,
  \label{eq: bayesian p.o.f}
\end{equation} 
where $\pi(\Theta |D) \propto l(D|\Theta) \pi(\Theta)$ is the posterior distribution of $\Theta$. The quantity in Equation \eqref{eq: bayesian p.o.f} is minimized over the product space $\mathcal{A}'$, the results are depicted in Figure \ref{fig: robust optimization over measure space}. The quantile of order $0.95$ is equal to $3.19 m$ which is slightly higher than for the maximal quantile over the moment class $\mathcal{A}$.

\begin{figure}[htb]
  \centering
  \includegraphics[scale=0.39]{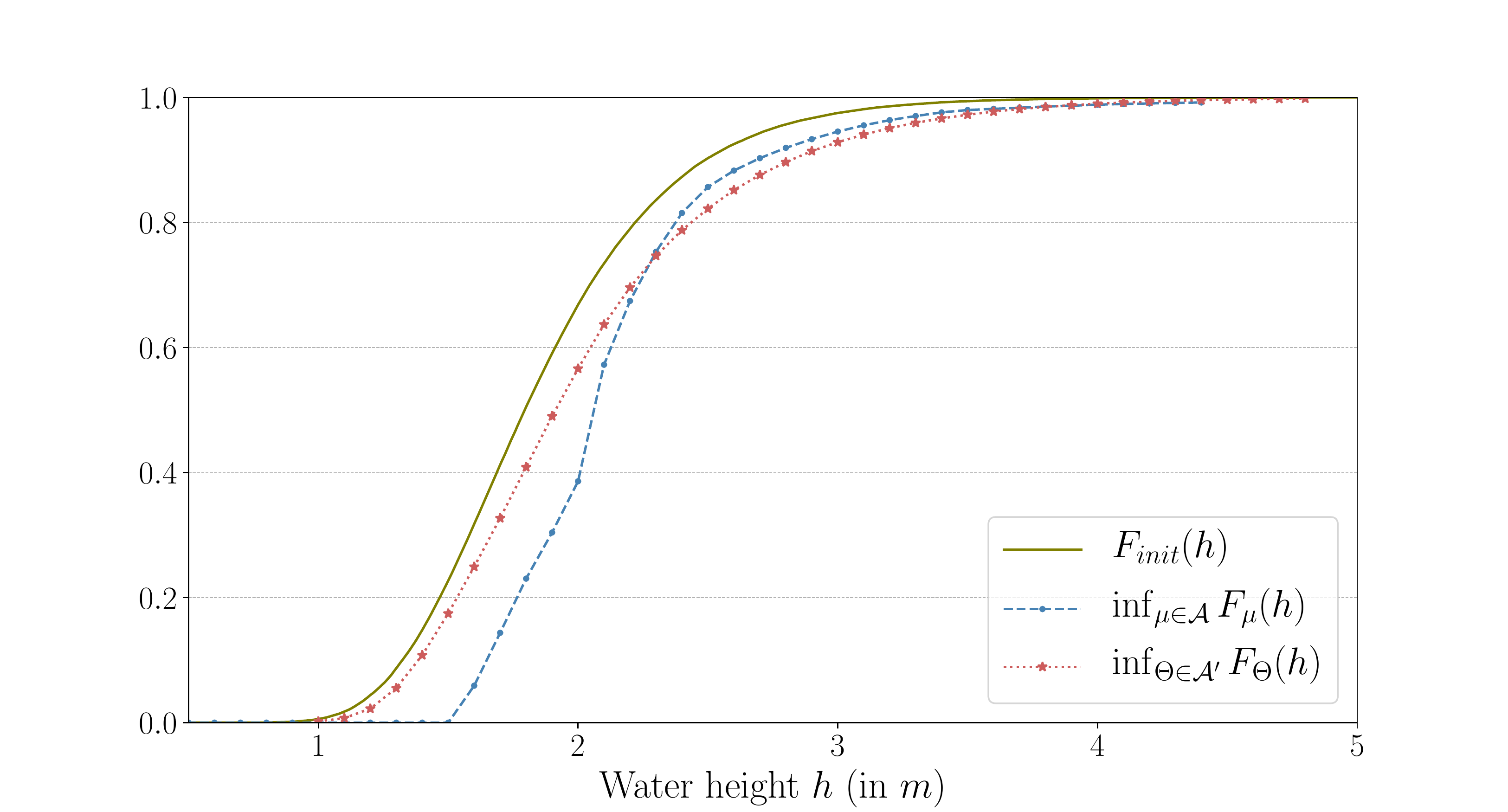}
  \caption{The solid line represents the CDF of the computer code $h\mapsto\mathbb{P}(H(J,K_s,Z_v,Z_m)\leq h)$ with the initial input distribution depicted in Table \ref{tab: Initial distribution hydraulic model}. The dashed line represents the CDF lowest envelop over the measure set $\mathcal{A}$ from Equation \eqref{eq: use case measure space definition}. The dotted line represents the optimization of the same quantity in a Bayesian framework when $J$ is a Gumbel distribution with prior density on its parameters.}
  \label{fig: robust optimization over measure space}
\end{figure}

\subsection{Computation of Sobol index}
In this section, we illustrate the impact of the uncertainty tainting the input distribution on the Sobol indices. We propose different robust computations of the Sobol' indices which lead to different interpretations. Each parameter $\mu_i$ belongs to a measure class $\mathcal{A}_i$ presented in Table \ref{tab : Constraints for hydraulic model}.

The first order indices $(S_i^0)_{1\leq i\leq 4}$ are classically computed with the nominal input distributions in Table \ref{tab: Initial distribution hydraulic model}.
We compute their robust version, corresponding to the bounds of $S_i$ when $\mu_i$ belongs to $\mathcal{A}_i$, that is 
\[S_i^{+} = \sup_{\mu_i\in\mathcal{A}_i} S_i \quad \text{ and }\quad  S_i^{-} = \inf_{\mu_i\in\mathcal{A}_i} S_i\ .\]
Notice that only the $i$th distribution varies in its moment class when computing bounds on $S_i$, the other distributions are set to their nominal choice in Table \ref{tab: Initial distribution hydraulic model}. Thus, $S_i^+$ represents the maximal contribution of the $i$th input alone onto the output variance, considering the uncertainty of the $i$th parameter distribution only. Note that $\sum_i S_i^+$ is not necessarily equals to $1$. The same interpretation holds for $S_i^{-}$. We also define the total indices \cite{iooss_review_2015} 
\begin{equation}
  \label{eq: total order}
  S_{Ti} =  \frac{\mathbb{E}_{\sim i}[\var_{\mu_i}(Y \ | \ X_{\sim i})]}{\var(Y)}    \ .
\end{equation}
Bounds on the $i$th total order index $S_{Ti}$ are computed with respect to the uncertainty of all the input distributions except the $i$th as follows:
\[S_{Ti}^{+} = \sup_{\substack{\mu_j\in\mathcal{A}_j \\ j\neq i}} S_{Ti}\quad\text{ and }\quad S_{Ti}^{-} = \inf_{\substack{\mu_j\in\mathcal{A}_j \\ j\neq i}} S_{Ti} \ .\] 
In other words, the bounds obtained are interpreted as the minimal and maximal total order index that the $i$th~input can have, considering the lack of knowledge in all but the $i$th input distribution. Hence, the lower bound on the $i$th total effect index corresponds to the minimal variance caused by its interaction between the $i$th input and the remaining parameters, and the upper bound on $S_{Ti}$ gives the maximal variance caused by its interaction between the $i$th input and the other parameters. 
\begin{figure}
  \centering
  \includegraphics[scale=0.8]{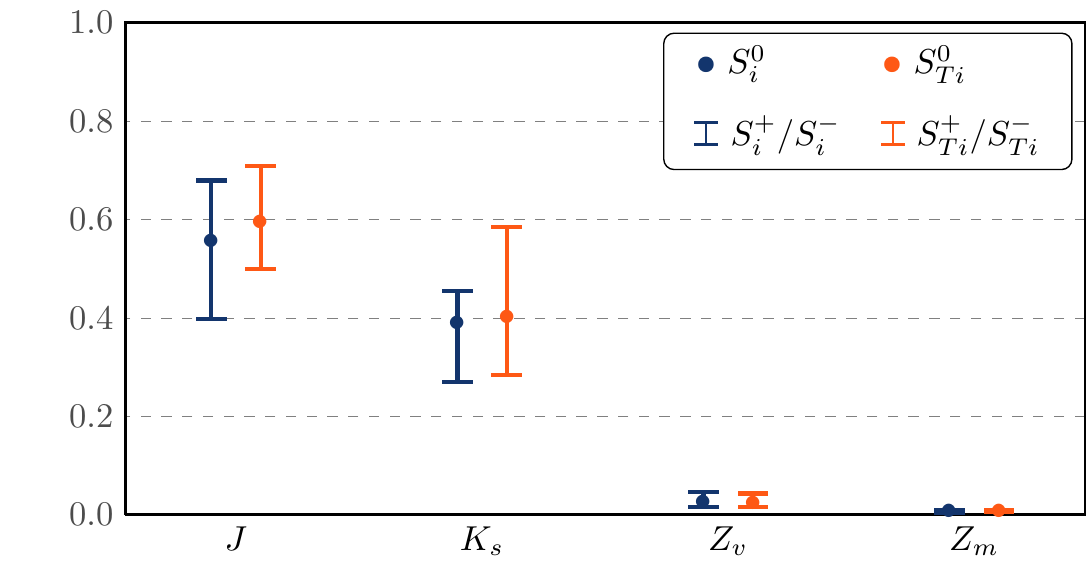}
  \caption{Different definitions of robustness for the Sobol indices yield different bounds.}
  \label{fig: Sobol robust}
\end{figure}

The results are displayed in Figure \ref{fig: Sobol robust}. We can see that whatever the input distribution is, the parameters $Z_v$ and $Z_m$ can be considered as weakly influent. The reference Sobol indices $S_i^0$ and $S_{Ti}^0$ indicate that $J$ is more influent than $K_s$. However, the optimization results show that their contribution to the output variance vary when accounting for their distribution uncertainty. The lower bound $S_1^{-}$ shows that for an imprecise measure $\mu_1$ varying in the moment class $\mathcal{A}_1^\ast$, the most penalizing distribution corresponds to a first order Sobol' index equals to $S_J^{-} = 0.39$. On the other hand, the upper bounds $S_{2}^{+}$ shows that there is a distribution in $\mu_{2}$ in $\mathcal{A}_2^\dagger$ corresponding to a first order Sobol' index $S_{2}^{+} = 0.55$. Therefore, considering the uncertainty affecting the distributions $\mu_{2}$ and $\mu_{1}$, the parameter $K_s$ can be in some cases more influent than $J$. It could be valuable in this situation to refine the information on these parameters in order to reduce the uncertainty tainting their distributions.

\section{Reduction theorem proof}
\label{sec: proofs}
In this section, we develop the proof of Theorem \ref{th: reduction theorem on product measure space}. We make few assumptions on the nature of the optimization space. In this way, the framework stays very general and can be extended to many different spaces, even though measure spaces constitute the main application of this paper. We first develop the reduction theorem for simple topological spaces before extending the result to product spaces. The proofs are quite short and rely only on simple topological arguments. We enlighten the assumption made in this work, in particular we compare our lower semicontinuity assumption with the Bauer maximum principle's upper semicontinuity one.
\subsection{Preliminary results}
The following two Lemmas are of great importance and gather the main arguments of our demonstration.
\begin{lemma}
    Let $\mathcal{A}$ be a convex subset of a locally convex topological vector space $\Omega$. If any point $x\in\mathcal{A}$ is the barycenter of some probability measure $\nu$ supported on $\Delta\subset\mathcal{A}$, then $\mathcal{A}\subset \overline{\textrm{co}}(\Delta)$.
    \label{lem: krein milman barycentric representation}
\end{lemma} 
\begin{proof}
    Let $K=\overline{\textrm{co}}(\Delta)$. We suppose that there exists $x_0 \in \mathcal{A}\backslash K$. By applying the Hahn-Banach separation theorem, there exists a continuous linear map $l:\Omega\rightarrow\mathbb{R}$, such that $\sup_{x\in K} l(x) < C < l(x_0)$, for some real $C$. The lower level set 
    \[ Z=\{x\in\mathcal{A}\ | \ l(x)\leq C\}\]
    obviously contains $\Delta$. Let $\nu_0$ be the representative measure of $x_0$, supported on $\Delta$ so that $\nu_0(Z) =1$. Then,
    \[l(x_0)=\int_Z l d\nu_0 \leq C < l(x_0)\ , \]
    leading to a contradiction.
\end{proof}

The next Lemma expresses the supremum of a quasi-convex function on the closed convex hull of some subset \cite{bereanu_quasi-convexity_1972}.
\begin{lemma}
    Let $\mathcal{A}$ be a convex set of a locally convex topological vector space. And let $f:\mathcal{A}\rightarrow\mathbb{R}$ be a quasi-convex lower semicontinuous function. Let $Y$ be a subset of $\mathcal{A}$ and ${\overline{\textnormal{co}}(Y)}$ its closed convex hull. Then 
    \[ \sup_{\overline{\textnormal{co}}({Y})} f(x) = \sup_{{Y}} f(x)\ , \]
    \label{lem: supremum quasiconvex function on convex hull}
\end{lemma}
\begin{proof}
    If  $\sup_{{Y}} f(x) = \infty$, there is nothing to prove. So, we may assume that $a := \sup_{{Y}} f(x)$ is finite. Let $Z_a=\{x\in\mathcal{A} \ | \ f(x) \leq a \}$. \\
    Obviously, we have $Y\subset Z_a$. But $Z_a$ is convex as $f$ is quasi-convex. Further, it is closed as $f$ is lower semicontinuous. Therefore, we have $\overline{\textrm{co}}({Y}) \subset Z_a$ because of the minimal property of the closed convex hull. Hence,
    \[\sup_{\overline{\textrm{co}}({Y})} f(x) \leq  \sup_{Z_a} f(x) \leq a = \sup_{{Y}} f(x)\ , \]
    The converse is obvious.
\end{proof}

It is remarkable that we assume the lower semicontinuity of the function to maximize. In contrast, the upper semicontinuity required in the Bauer maximum principle is a more standard assumption for function maximization. The proof of Lemma \ref{lem: supremum quasiconvex function on convex hull} clarifies how our assumption acts. Indeed, the lower semicontinuity is used to enforce the closure of the set $Z_a = \{ x\in\mathcal{A} \ | \ f(x)\leq \sup_Y f(x) \}$. This argument differs from Choquet's demonstration of the Bauer maximum principle \cite[p.102]{choquet_lectures_1969}. In the latter,  the author studies the closure of the set $\{ x\in\mathcal{A} \ | \ f(x) =  \sup_Y f(x) \}$ for an upper semicontinuous function $f$ on a compact space. Doing so, the assumptions of compactness and upper semicontinuity in the Bauer maximum principle are used in order to show that the optimum of the function $f$ is reached. This is not needed here. 

From Lemma \ref{lem: krein milman barycentric representation} and \ref{lem: supremum quasiconvex function on convex hull}, we establish the next Theorem. It is analogous to the Bauer maximum principle, where the assumption of compactness is replaced by that of integral representation. The integral representation is always satisfied on compact sets, thanks to the Choquet representation theorem \cite[p.153]{choquet_lectures_1969}. So that the next theorem is analogous to the Bauer maximum principle under a compactness assumption. Hence, our integral representation assumption is, in a way, more general.
\begin{theorem}
    Let $\mathcal{A}$ be a convex subset of a locally convex topological vector space $\mathcal{X}$. We assume that any point $x\in\mathcal{A}$ is the barycenter of a probability measure $\nu$ supported on $\Delta \subset \mathcal{A}$. Let $f: \mathcal{A}\rightarrow\mathbb{R}$ be a quasi-convex lower semicontinuous function. 
    Then \[ \sup_{x\in\mathcal{A}} f(x) = \sup_{x\in\Delta} f(x) \ .\]
    \label{th: generalized non compact bauer representation}
\end{theorem}
\begin{proof}
    From Lemma \ref{lem: krein milman barycentric representation} \[\mathcal{A}\subset\overline{\textrm{co}}(\Delta)\ .\]
    Then applying Lemma \ref{lem: supremum quasiconvex function on convex hull} on the lower semicontinuous quasi-convex function $f$, we obtain
    \[\sup_{\mathcal{A}} f(x) \leq \sup_{\overline{\textrm{co}}(\Delta)} f(x) = \sup_{\Delta}f(x)\ .\]
    The converse inequality holds obviously.
\end{proof}

We may relate our result to the one of Vesely \cite{vesely_jensen_2017}, who proves that Jensen's integral inequality remains true for a lower semicontinuous convex function on a convex set in a locally convex topological vector space.
Indeed, given a probability measure $\nu$ supported on $\Delta$ with barycenter $x_\nu$ and a convex \textit{lsc} function $f$, we get thanks to Jensen's inequality:
\[ f(x_\nu) = f\left(\int_{\Delta} x d\nu\right) \leq \int_{\Delta} f(x) d\nu \leq \sup_{\Delta} f(x)\ . \] 
Therefore, Theorem \ref{th: generalized non compact bauer representation} can be seen as some extremal version of Jensen's integral inequality. Moreover, it extends to quasi-convex functions which are more general than convex functions.

\subsection{Extension to Product spaces}
The following Theorem shows that the optimum of a quasi-convex lower semicontinuous function on a product space may be computed on the generator of the optimization set.
\begin{theorem}
    Let $\mathcal{A}_i$ be a convex subset of a locally convex topological vector space $\Omega_i,\, 1\leq i\leq n$. We assume that any point $x_i\in\mathcal{A}_i$ is the barycenter of some probability measure $\nu_i$ supported on $\Delta_{i} \subset \mathcal{A}_i$. We equip $\mathcal{A}:=\prod_{i=1}^n \mathcal{A}_i$ with the product topology, we suppose further that $f: \mathcal{A}\rightarrow\mathbb{R}$ is marginally quasi-convex lower semicontinuous. 
    Then \[ \sup_{\substack{{x_i\in\mathcal{A}_i} \\ {1\leq i\leq n}}} f((x_1, \dots, x_n)) = \sup_{\substack{{x_i\in\Delta_{i}} \\ {1\leq i\leq n}}} f((x_1, \dots, x_n))\ . \]
    \label{th: generalized non compact product bauer representation}
\end{theorem}
\begin{proof}
    For the sake of simplicity, we only develop the case where $i=2$. The general case follows the same lines. Then, 
    \[ \sup_{\substack{{x_1\in\mathcal{A}_1} \\ {x_2\in\mathcal{A}_2}}} f((x_1,x_2)) = \sup_{x_1\in\mathcal{A}_1} \sup_{x_2\in\mathcal{A}_2} f((x_1,x_2))\ . \]
    Now, the map $x_2\mapsto f((x_1,x_2))$ is a quasi-convex lower semicontinuous function, by applying Theorem \ref{th: generalized non compact product bauer representation}, it follows that 
    \[\sup_{x_2\in\mathcal{A}_2} f((x_1,x_2)) = \sup_{x_2\in\Delta_2} f((x_1,x_2)), \text{ for every } x_1\in\mathcal{A}_1\ .\]
    Therefore, 
    \begin{IEEEeqnarray*}{rCl} 
      \sup_{x_1\in\mathcal{A}_1} \sup_{x_2\in\mathcal{A}_2} f((x_1,x_2)) & = &  \sup_{x_1\in\mathcal{A}_1} \sup_{x_2\in\Delta_2} f((x_1,x_2)) , \\
      & = &  \sup_{x_2\in\Delta_2}  \sup_{x_1\in\mathcal{A}_1}f((x_1,x_2)) , \\
      & = & \sup_{x_2\in\Delta_2}  \sup_{x_1\in\Delta_1} f((x_1,x_2)) ,
    \end{IEEEeqnarray*}
    by applying the same reasoning to $x_1\mapsto f((x_1,x_2))$.
\end{proof}
\begin{remark}
    If a function is upper semicontinuous and quasi-concave, the same reduction applies to the minimum. 
\end{remark}

\subsection{Proof of the main result}
\label{subsec: proof of main result}
Theorem \ref{th: reduction theorem on product measure space} is slightly more complicated than Theorem \ref{th: generalized non compact product bauer representation}, so that we detail its proof. Let $\mathcal{X}, \mathcal{A}$ and $\Delta$ be as defined in \Cref{subsec: construction of the product measure spaces}. 
\begin{proof}[Proof of Theorem \ref{th: reduction theorem on product measure space}] 
  We recall that $\mathcal{A}:=\prod_{i=1}^d \mathcal{A}_i$ is a product of measure spaces, where $\mathcal{A}_i$ is either a moment space or a unimodal moment space. Therefore, each measure $\mu_i \in\mathcal{A}_i$ satisfies $N_i$ moment constraints. More precisely, for measurable functions $\varphi^{(j)}_i:X_i\rightarrow \mathbb{R}$, we have $\mathbb{E}_{\mu_i}[\varphi^{(j)}_i]\leq 0$ for $1\leq j\leq N_i$ and $1\leq i\leq d$. 
  
  Moreover, in Theorem \ref{th: reduction theorem on product measure space}, we also enforce constraints on the product measure $\mu=\mu_1 \otimes \dots \otimes \mu_d \in\mathcal{A}$, such that, for measurable functions $\varphi^{(j)}:\mathcal{X}\rightarrow \mathbb{R}$, $1\leq j\leq N$, we have $\mathbb{E}_\mu[\varphi^{(j)}] \leq 0$. 

  Let $f$ be a marginally quasi-convex lower semicontinuous function. Then,
  \[\sup_{\substack{\mu_1,\dots ,\mu_d \in \otimes\mathcal{A}_i \\ \mathbb{E}_{\mu_1,\dots,\mu_d}[\varphi^{(j)}] \leq 0 \\ 1\leq j\leq N}} f(\mu)  = \sup_{\mu_1\in \mathcal{A}_1} \dots \sup_{\mu_{d-1} \in\mathcal{A}_{d-1}} \sup_{\substack{\mu_d \in \mathcal{A}_d \\ \mathbb{E}_{\mu_1,\dots,\mu_d}[\varphi^{(j)}] \leq 0 \\ 1\leq j\leq N}} f(\mu_1, \dots, \mu_d)\ .\]
  Now, for fixed $\mu_1, \dots \mu_{d-1}$, we have that 
  \[\mathbb{E}_{\mu_1,\dots,\mu_d}[\varphi^{(j)}] = \mathbb{E}_{\mu_d}\left[\mathbb{E}_{\mu_1,\dots,\mu_{d-1}}[\varphi^{(j)}_i]\right] \leq 0 \ ,\]
  for $1\leq j\leq N$, which are moment constraints on the measure $\mu_d$. This means that there are in total $(N_d + N)$ moment constraints enforced on $\mu_d$. Therefore, $\mu_d$ has an integral representation supported on the set of convex combination of $N_d+N+1$ extreme points (which are either Dirac masses, or uniform distributions). Hence, for fixed $\mu_1,\dots, \mu_{d-1}$, and because the function $\mu_d \longmapsto f(\mu_1, \dots, \mu_d)$ is quasi-convex and lower semicontinuous. We have from Theorem \ref{th: generalized non compact bauer representation} 
  \[\sup_{\substack{\mu_d \in \mathcal{A}_d \\ \mathbb{E}_{\mu_1,\dots,\mu_d}[\varphi^{(j)}] \leq 0 \\ 1\leq j\leq N}} f(\mu_1, \dots, \mu_d) = \sup_{\substack{\mu_d \in \Delta_{N_d+N} \\ \mathbb{E}_{\mu_1,\dots,\mu_d}[\varphi^{(j)}] \leq 0 \\ 1\leq j\leq N}} f(\mu_1, \dots, \mu_d) .\]
  So that,
  \[\sup_{\substack{\mu_1,\dots ,\mu_d \in \otimes\mathcal{A}_i \\ \mathbb{E}_{\mu_1,\dots,\mu_d}[\varphi^{(j)}] \leq 0 \\ 1\leq j\leq N}} f(\mu)  = \sup_{\substack{\mu_1,\dots ,\mu_d \in \mathcal{A}_1\otimes\dots \otimes \mathcal{A}_{d-1} \otimes \Delta_d \\ \mathbb{E}_{\mu_1,\dots,\mu_d}[\varphi^{(j)}] \leq 0 \\ 1\leq j\leq N}} f(\mu) \ . \]
  Consequently, the last component of $\mu$ can be replaced by some element of $\Delta_{N_i+N}$. By repeating this argument to the other components, the result follows.
\end{proof}

\section{Conclusion}
\label{sec: conclusions}

We have studied the optimization of a quasi-convex lower semicontinuous function over a set of product of measure spaces $\mathcal{A}=\prod_{i=1}^{d} \mathcal{A}_i$. Specific product measures sets are studied: the product of moment classes or unimodal moment classes. In those classes, we have an integral representation on the extreme points $\Delta_i$, that are either finite mixture of Dirac masses or finite mixture of uniform distributions. This integral representation can be seen as a non-compact form of the Krein-Milmann theorem. We have shown that the optimization of a quasi-convex lower semicontinuous function on the product space $\mathcal{A}$ is reduced to the $d$-fold product of finite convex combination of extreme points of $\mathcal{A}_i$.

This powerful Theorem provides numerous industrial applications. We develop for example the optimization of the quantile of the output of a computer code whose input distributions belong to measure spaces \cite{stenger_optimal_2020}. We also highlight through several illustrated applications how our framework generalizes both Optimal Uncertainty Quantification \cite{owhadi_optimal_2013} and robust Bayesian analysis \cite{insua_robust_2000,berger_robust_1990}.

Although, we have an explicit representation of the extreme points, the optimization is not obvious because of the high number of \textit{generalized} moment constraints enforced. In \cite{stenger_optimal_2020}, the authors propose an original parameterization of the problem in the presence of \textit{classical} moment constraints, allowing fast computation of the quantities of interest presented in \Cref{sec: applications}.

The product of measure spaces reflects the mutual independence of the model parameters. In case of a dependence structure, the Lasserre hierarchy of relaxation in semidefinite programming \cite{lasserre_moments_2010} provides an alternative solution for practical optimization.

\appendix
\section{Proof of Theorem \ref{th: extreme points of unimodal set}} 
\begin{proof}
  The proof is quite technical and as it is not the main topic of our paper, details are kept to the bare minimum. We mainly gather different results to prove our point, so that the interested reader might refer to it. Letting $X$ be an interval of $\mathbb{R}$, $\mathcal{H}_a(X)$ is the set of all probability measures on $x$ which are unimodal at $a\in X$. From \cite[p.19]{bertin_unimodality_1997} we now that $\mathcal{H}_a(X)$ is a simplex, meaning that every probability measure in $\mathcal{H}_a$ is the barycenter of a unique probability measure supported on $\mathcal{U}_a(X)$. Choquet \cite[p.160]{choquet_lectures_1969} defines another type of simplex, named Choquet simplex. A convex subset $K$ of a locally convex topological vector space is called a Choquet simplex if and only if the cone $\tilde{K} = \{(\lambda x, \lambda): x\in K, \lambda >0\}$ is a lattice cone in its own order (that is, the vector space $\textrm{span}(\tilde{K})$ is a lattice when its positive cone is taken to be $\tilde{K}$). The important point is that these two definitions are connected, and it holds from \cite[p.47]{Winkler_simplices_1985} that each simplex is a Choquet simplex. Moreover, in finite dimensional compact sets these two definitions are equivalent. Therefore, $\mathcal{H}_a(X)$ is also a Choquet simplex. We now define
  \begin{IEEEeqnarray*}{rCl}
    K & = & \{ \mu\in\mathcal{H}_a(X): \varphi_1,\dots, \varphi_n $ are $ \mu$-integrable$\}\ , \\
    F(\mu) & = &    \left(\int \varphi_1\, d\mu, \dots, \int\varphi_n \, d\mu\right)\ , \\
    W & = & F[K]\cap \prod_{i=1}^n ]-\infty, 0] $ or $W=F[K]\cap \{(0,\dots, 0)\} \ .
  \end{IEEEeqnarray*}
  It is already known that the extreme points set of $\mathcal{H}_a(X)$ is precisely $\mathcal{U}_a(X)$ as shown in \cite[p.19]{bertin_unimodality_1997}. Indeed, a classical result due to Khintchine \cite{khintchine1938unimodal} states that uniform distributions constitute the \textit{elementary units} of the set of all unimodal probability measures. However, we wish to known how the extreme points of $\mathcal{H}_a(X)$ also characterize the extreme points of a convex subset of $\mathcal{H}_a(X)$ defined as $\mathcal{A}^\dagger= F^{-1}[W] = \{\mu\in\mathcal{H}_a(X) \ | \ \mathbb{E}_\mu[\varphi_i]\leq 0, \ 1 \leq i \leq n\}$. Consequently, we wish to apply \cite[Proposition 2.1]{winkler_extreme_1988} which states that the set $\mathcal{A}^\dagger$ satisfies
  \begin{equation}
    \label{eq: appendix extreme points unimodal}
    \text{ex}\{\mathcal{A}^\dagger\} \subset \Delta^\dagger(n)  = \left\lbrace \mu\in\mathcal{A}^\dagger \ | \ \mu = \sum_{i=1}^{n+1} \omega_i \mathfrak{u}_i, \omega_i\geq 0 , \mathfrak{u}_i\in\mathcal{U}_a(X)\right\rbrace\ .
  \end{equation}
  However, in order to apply  \cite[Proposition 2.1]{winkler_extreme_1988} it remains to check that $K$ is linearly compact, meaning each of its line meets $K$ in a compact interval. By the main theorem in \cite{kendall_simplexes_1962}, it is sufficient to show that $\mathbb{R}_+\cdot K$ is a lattice cone in its own order. Indeed, condition ($2^0$) in the main theorem is an equivalent formulation of linear compactness as shown in the same reference on p.369. Of course, the cone $\mathbb{R}_+\cdot \mathcal{H}_a(X)$ is a lattice cone in its own order because it is a Choquet simplex. Now, choose $\mu\in\mathbb{R}_+\cdot K$ and $\nu\in \mathbb{R}_+\cdot \mathcal{H}_a(X)$ such that $(\mu-\nu) \in \mathbb{R}_+\cdot K$, then $\nu\in\mathbb{R}_+\cdot K$ since 
  \[\int |\varphi_i|\, d\nu\leq \int |\varphi_i| \,d\mu \quad \text{ for every } i=1,\dots, n. \]
  Hence, $\mathbb{R}_+\cdot K$ is a hereditary subcone of $\mathbb{R}_+\cdot \mathcal{H}_a(X)$ and consequently a lattice cone in its own order. This proves that $K$ is linearly compact and that \cite[Proposition 2.1]{winkler_extreme_1988} applies. As stated, it follows that the set $\mathcal{A}^\dagger = F^{-1}[W] = \{\mu\in\mathcal{H}_a(X) \ | \ \mathbb{E}_\mu[\varphi_i]\leq 0, \ 1 \leq i \leq n\}$ satisfies Equation \eqref{eq: appendix extreme points unimodal}.

  Now that the extreme points of $\mathcal{A}^\dagger$ are classified and observing that this set is closed with respect to the weak topology; Corollary 1 in \cite{weizsacker_integral_1979} concludes that every measure in $\mathcal{A}^\dagger$ has an integral representation supported on $\Delta^\dagger(n)$.
\end{proof}
\vspace{0.4cm}
\noindent\textbf{Acknowledgement.} We are grateful to the two reviewers who greatly help to improve the manuscript. Support from the ANR-3IA Artificial and Natural Intelligence Toulouse
Institute is gratefully acknowledge.

\bibliographystyle{siamplain}
\bibliography{references}
\end{document}